\documentclass[a4paper,12pt]{article}
\usepackage{cite}
\usepackage{amsmath,amssymb}
\usepackage{algorithmic}
\usepackage{float}
\usepackage[margin=1in]{geometry}
\usepackage{graphicx}
\usepackage{textcomp}
\usepackage{xcolor}
\usepackage[utf8]{inputenc}
\usepackage[english]{babel}
\usepackage{color}
\usepackage{amsfonts,enumitem}
\usepackage{amssymb}
\usepackage{hyperref} 
\usepackage{graphicx,tikz,todonotes}
\newtheorem{theorem}{Theorem}[section]
\newtheorem{corollary}{Corollary}[theorem]
\newtheorem{lemma}[theorem]{Lemma}
\newtheorem{proposition}{Proposition}

\newtheorem{remark}{Remark}
\newtheorem{definition}{Definition}
\usepackage[utf8]{inputenc}
\usepackage[T1]{fontenc}

\def\BibTeX{{\rm B\kern-.02em{\sci\kern-.010em b}\kern-.01em
    T\kern-.1667em\lower.7ex\hbox{E}\kern-.110emX}}
    
    \newenvironment{proof}[1][]{\noindent {\bf Proof #1:\;}}{\hfill $\Box$}

\begin{document}

\title{Intrinsic Derivation of the  Equations of a Snake Robot based on a Cosserat Beam Model\thanks{This is the report of an M2 internship  held by the first author at Universit\'e Paris Saclay} \\
}

\author{Anis Bousclet\and Fr\'ed\'eric Boyer\thanks{IMT-Atlantique, France, frederic.boyer@imt-atlantique.fr}\and Yacine Chitour\thanks{L2S, Universit\'e Paris Saclay and CentraleSup\'elec, France, yacine.chitour@l2s.centralesupelec.fr}\and Swann Marx\thanks{LS2N, ECN \& CNRS, France, swann.marx@ls2n.fr}}

\maketitle
\tableofcontents

\begin{abstract}
In this paper, we present an intrinsic derivation of the equations ruling the dynamics motion of a snake robot dynamics. Based on a Cosserat beam model, we first show that the extended configuration space is a Lie group. Endowing it with an appropriate left invariant metric, the corresponding  Euler-Poincar\'e equations can be reduced to a system of hyperbolic PDEs  in the Lie algebra $\mathfrak{se}(3)$. We also provide the constitutive law describing the actuation in this system of PDEs. 
\end{abstract}


\section{Introduction}
Until the beginning of the last century, robotics was all about producing rigid systems. The credo of classical robotics was "the stiffer, the better". In this context, the model used was that of rigid multi-body systems whose dynamic equations are trivially given by Lagrange's equations on the configuration space of the robot's joint coordinates. Quickly produced by the robotics community, these equations are ODEs or at most DAEs, whose numerical resolution and exploitation for the synthesis of dynamic control laws are now well known. In the early 2000s, robotics underwent a paradigm shift. The inadequacy of our machines compared to the performance of even the simplest-looking animals led to the emergence of a new robotics approach, inspired by the living world. Contrary to the "stiffer is better" approach of classical robotics, the idea behind this revolution is to design robots with deformable bodies, capable of interacting mechanically in an "optimal" way with their environment. In this new kind of robotics, commonly referred to as continuous or soft robotics, snakes are the emblematic animals for the performance sought, and one of the most inspiring models for roboticists. Without any lateral limbs, these animals are able to control their deformation and internal stresses so as to move in virtually any environment, including the most challenging, such as the constrained environments of the canopy or the soft terrain of deserts. From this point of view, it is vital to study these animals and their robotic artifacts in order to design new mechanical architectures and their associated control laws. Given their prohibitive number of degrees of freedom, models based on the mechanics of continuous media undergoing large transformations (displacement + deformations) are the obvious alternative to the discrete models of rigid robotics. Based on this idea, the Cosserat beam model was introduced to the field of robotics in \cite{BPK-2006} (see also \cite{Rucker-2011,RBDS-2018} and references therein) and has gradually established itself as one of the dominant modeling paradigms for continuous robotics. In the Cosserat model, a beam, or snake, is seen as a continuous stack of infinitely many rigid sections (vertebrae) whose pose (orientation-position) is naturally described by a transformation of the Lie group $SO(3)\times \mathbb{R}^3$ or 
$\mathfrak{se}(3)$). 

Taking up the work of the Cosserat brothers \cite{Cosserat}, the idea is to take advantage of the intrinsic formulations of geometrical mechanics to relieve the snakes' dynamic model of extrinsic nonlinearities due to the parameterization of their vertebral poses. This idea was deployed in \cite{BR-2016} where the snake robot is seen as a continuum distribution of rigid bodies, and the configuration space is furnished with a Lagrangian. erive the PDEs. The Euler-Poincar\'e reduction was adapted to these systems to produce a set of PDEs called Poincar\'e-Cosserat equations. Today, this modeling paradigm is being exploited for non-invasive surgical robotics via numerical modeling and simulation techniques adapted to the needs of this community (see \cite{BLCR-2021, BLCRA-2023} and references therein). 

Nevertheless, the objectives of this work remain pragmatic, and the use of matrix groups and the numerous isomorphisms linking their algebra to the Euclidean space $\mathbb{R}^3$, hides a more intrinsic formulation of these equations, which control theory could take advantage of to discover the secrets of snake performance and reproduce them on the continuous robots of our laboratories. It is the aim of this note to produce such a formulation.

Recall that Cosserat mediums were introduced in \cite{Cosserat} and the idea is to consider additional degree of freedom that describe the evolution of microstructure, here we adopt the geometrical formulation of Cosserat medium exposed in \cite{Epstein}.

The structure of the note goes as follows. In section 2, we begin by introducing the state space that is called in mechanics phase space and which is derived from the configuration space.  We will actually consider a slightly bigger space that can be identified to a group where we can translate equations written in a nonlinear space onto an euclidean space. In section 3, 
we introduce the Lagrangian in order to derive the equations of motion, as the difference between the kinetic energy, which induces a Riemannian structure on the group, and a potential energy, which encodes the elasticity of the vertebrate of the snake. In section 4, we first apply  the principle of least action to obtain a variational formulation of the motion of the snake and in section 5, we explain how to complement them in order to derive the controlled equations. 

In the appendix we give intuitive explanation of the insufficiency of classical medium, and the formula of kinetic energy of a Cosserat medium, and we recall some elementary notions in algebra and differential geometry.

\section{The configuration space}
We next provide a rigorous definition of a Cosserat medium that models the motion of snake robot such that sections have rigid body motion, the reader can see in Appendix why the classical approach fails.  
\begin{definition}
A Kinematic Cosserat medium is modeled by a trivial principal bundle $F\mathcal{L}=[0,1]\times \mathrm{SO}(3)$ where $\mathcal{L}=[0,1]$ is the basis, $\pi:F\mathcal{L}\rightarrow{\mathcal{L}}$ is the projection $\pi(z,r)=z$.
\end{definition}
We call $Emb_{c}(\mathcal{L})$ the set of all embedding of $\mathcal{L}$ in $\mathbb{R}^{3}$ into the Cosserat sens, this require the existence of an embedding of $\mathcal{L}$, and of a map between the two principal bundles that is right invariant.
\begin{definition}
A configuration of the Cosserat medium is a map $P:F\mathcal{L}\rightarrow{F\mathbb{R}^{3}}$ such that there exist an embedding $p:\mathcal{L}\rightarrow{\mathbb{R}^{3}}$ with the property $\pi_{\mathbb{R}^{3}}\circ P=p\circ\pi_{\mathcal{B}}$, we denote $P(z,r)=(p(z),Q_{z}(r))$.

We suppose also that $P(z,qr)=(p(z),Q_{z}(q)r)$ for all $(z,q),(z,r)\in F\mathcal{L}$.
\end{definition}
Poincar\'e has observed in [6] that if there exists a Lie group that acts transitively on the configuration space, we obtain simpler equations, so the idea is to let a Lie group act on $Emb_{c}(\mathcal{L})$.
We consider a reference configuration $P_{0}$, for each element $\mathbf{g}\in G=C^{\infty}(\mathcal{L},\mathrm{SE}(3))$ we get another "configuration" $\mathbf{g}*P_{0}$, the point is that if we don't deform so much the reference configuration, we will get a new configuration, but if we deform very much, we can get a "configuration" such that $p$ is not an embedding.

This suggests to work on the space of deformations that allow the snake to self cut himself. 
\begin{definition}
The space $M=C^{\infty}(F\mathcal{L},F\mathbb{R}^{3})$ is the set of maps $P:F\mathcal{L}\rightarrow{\mathbb{R}^{3}}$ such that there exists a smooth map $p:\mathcal{L}\rightarrow{\mathbb{R}^{3}}$ and satisfies $\pi\circ P=p\circ\pi$ and $P(z,qr)=(p(z),Q_{z}(q)r)$. 
\end{definition}
Before going further, let us see the structure of $G$ and some properties of this Lie group.
\begin{definition}
The group $G=C^{\infty}(\mathcal{L},\mathrm{SE}(3))$ is a Lie group of infinite dimension, an element of $G$ is a smooth function $g:[0,1]\rightarrow{SE(3)}$, we denote it by 
$\mathbf{g}=g(.)$.  
the neutral element of the group is the constant function $e_{G}(z)=\begin{pmatrix}I_{3}&0\\0 &1\end{pmatrix}$, the multiplication and inversion is done by fixing $z\in\mathcal{L}$ $$\mathbf{g_{1}g_{2}}(z)=g_{1}(z).g_{2}(z),$$
and $$\mathbf{g}^{-1}(z)=(g(z))^{-1}.$$
The Lie algebra is $\mathfrak{G}=C^{\infty}(\mathcal{L},\mathfrak{se}(3))$ and 
the tangent space at $\mathbf{g}\in G$ is given by 
$$
T_{\mathbf{g}}G=\left\{\delta \mathbf{g}\in C^{\infty}(\mathcal{L},TSE(3)) |\forall{z\in\mathcal{L}}, \delta\mathbf{g}(z)\in T_{\mathbf{g}(z)}SE(3)\right\}.
$$

The adjoint representation of $G$ is 
$$Ad_{\mathbf{g}}(\mathbf{V})(z)=Ad_{\mathbf{g}(z)}(\mathbf{V}(z)),$$
$\forall{\mathbf{g}\in G},\forall{\mathbf{V}\in\mathfrak{G}}$, and the adjoint representation of $\mathfrak{g}$ is 
$$ad_{\mathbf{V}}\mathbf{W}(z)=[\mathbf{V}(z),\mathbf{W}(z)].$$
When we have an evolution $g:(t,z)\in\mathbb{R}^{+}\times[0,1]\rightarrow{g(t,z)\in SE(3)}$, for all $t\geq 0$ we design by $\mathbf{g}(t)$ the element of $G$ such that $\mathbf{g}(t)(z)=g(t,z)$. 

\end{definition}

\begin{proposition}
The right action $\Phi: M\times G\rightarrow{M}$ defined by 
\begin{equation}\label{eq:grpact}
\Phi(P,\mathbf g)(z,r)=(\mathbf{g}(z)^{-1}(p(z)),Q_{z}(R_{\mathbf{g}(z)}.r)),\end{equation}
is transitive and free. 
\end{proposition}
\begin{proof}
We show only that the action is free and transitive.
Let $\mathbf{g}\in G$ such that $\Phi(P,\mathbf{g})=P$, so $\mathbf{g}(z)^{-1}(p(z))=p(z)$ and $Q(R_{\mathbf{g}(z)}r)=Q(r)$, this implies that $\mathbf{g}(z)=I_{4}=\mathbf{e}_{G}$. 
Let $P_{1},P_{2}\in M$, we define $R_{\mathbf{g}}$ such that $Q^{1}_{z}(R_{\mathbf{g}(z)}r)=Q^{2}_{z}(r)$, and we deduce the translation part from $\mathbf{g}^{-1}(z)(p_{1}(z))=p_{2}(z)$.
\end{proof}

\vspace{2mm}

We now explicit the process that allow us to work on $G$ instead of $M$. 

We first, fix a reference configuration $p_{0}\in M$ that models our snake robot at the initial time.
Consider an evolution of the snake robot that is modeled by a curve $P(t)\in M$, using the fact that the action is free and transitive, the mapping $$\Phi(p_{0},.):\mathbf{g}\in G\rightarrow{\Phi(p_{0},\mathbf{g})\in M},$$ is a diffeomorphism, so for a curve $P(t)\in M$ we get a curve $\mathbf{g}(t)\in G$ that satisfies \begin{equation}\label{eq:reduction}
    \Phi(p_{0},\mathbf{g}(t))=P(t).\end{equation} \\

\section{Riemannian structure}
We consider now a Riemannian structure on the Lie group $G$ that is inherited from the kinetic energy, cf. \cite{Arnold1,Arnold2}. We refer the reader to the appendix for an intuitive explanation of the formula that gives the kinetic energy for Cosserat medium and the original paper of Cosserat medium \cite{Cosserat}.
\begin{definition}
A dynamical Cosserat medium is a kinematic Cosserat medium $\mathcal{L}$ furnished with a mass measure $m:\mathcal{L}\rightarrow{\mathbb{R}^{*}_{+}}$ and a a positive definite inertia operator $I_{z}:\mathbb{R}^{3}\rightarrow{\mathbb{R}^{3}}$ for each $z\in\mathcal{L}$ i.e $\forall{v\neq 0}$ $(v,I_{z}(v))>0$.
\end{definition}
An intuitive explanation of this definition is given in the appendix, $m$ is the mass measure of the vertebrate of the snake, and $I$ are the inertia tensor of the rigid sections.\\ \\
In order to compute the Riemannian metric we need a reference configuration $P_{0}\in M$.\\
The Riemannian metric is explicitly given by :
\begin{proposition}
For every $\mathbf{g}\in G$ the bilinear form at each tangent space $T_{\mathbf{g}}G$ 
\begin{equation}\label{eq:riemetric}
    \ll\delta \mathbf{g_{1}},\delta \mathbf{g_{2}}\gg_{g}=\int_{\mathcal{L}}m_{z}(\theta_{L}(\delta \mathbf{g_{1}})(z)(p_{0}(z)),\theta_{L}(\delta \mathbf{g_{2}})(z)(p_{0}(z)))+(\omega_{\theta_{L}(\delta \mathbf{g_{1}})(z)},I_{z}(\omega_{\theta_{L}(\delta \mathbf{g_{2}})(z)}))dz,
\end{equation} for each $\delta \mathbf{g_{1}},\delta \mathbf{g_{2}}\in T_{\mathbf{g}}G$ is a Riemannian metric on $G$, and it is left invariant that is
$$\ll\delta \mathbf{g_{1}},\delta \mathbf{g_{2}}\gg_{\mathbf{g}}=\ll\theta_{L}(\delta \mathbf{g_{1}}),\theta_{L}(\delta \mathbf{g_{2}})\gg_{\mathbf{e}}.$$
\end{proposition}
\begin{proof}
We show that $\ll,\gg_{e}$ is a scalar product on $\mathfrak{g}$.\\
The bi-linearity and positiveness are trivial, it remains to show the positive definiteness, let $\mathbf{V}\in\mathfrak{G}$ such that $$\ll V,V\gg_{e}=0,$$
this also reads $$\int_{\mathcal{L}}m_{z}|\mathbf{V}(z)(p_{0}(z))|^{2}+(\omega_{\mathbf{V}(z)},I_{z}(\omega_{\mathbf{V}(z)}))dz=0,$$
which implies by positiveness of $m$ and $I$ and continuity of the integrand that $\omega_{V(z)}=0$ and $\mathbf{V}(z)(p_{0}(z))=0$, which implies that $\mathbf{V}=0$. \end{proof}\\ \\
Notice that for each fixed $z\in\mathcal{L}$ we have a scalar product on $\mathfrak{se}(3)$ defined by $\forall{V,W\in\mathfrak{se}(3)}$ $$<V,W>_{z}=m_{z}(V(p_{0}(z)),W(p_{0}(z)))+(\omega_{V},I_{z}(\omega_{W})).$$
We consider a motion $P:\mathbb{R}^{+}\rightarrow{M}$, by the fact that the action \eqref{eq:grpact} is free and transitive, we get a curve $\mathbf{g}:\mathbb{R}^{+}\rightarrow{G}$ such that it satisfies \eqref{eq:reduction}, denoting $\mathbf{W}(t)=\theta_{L}(\mathbf{g'}(t))$, the formula of the kinetic energy is :  $$E_{k}(P(t),P'(t))=\frac{1}{2}\ll \mathbf{g}'(t),\mathbf{g}'(t)\gg_{g(t)},$$ for the motion $P(t)=P_{0}*\mathbf{g}(t)$.\\
An explicit expression of the kinetic energy is
\begin{equation}\label{eq:kinetic}
    E_{k}(P(t),P'(t))=\frac{1}{2}\int_{\mathcal{L}}m_{z}|\mathbf{W}(t)(z)(p_{0}(z))|^{2}+(\omega_{\mathbf{W}(t)(z)},I_{z}(\omega_{\mathbf{W}(t)(z)}))dz.
 \end{equation}
 The first term corresponds to the kinetic energy of the vertebrate, and the second one is the kinetic energy of the rigid sections. An intuitive presentation of this formula is given in the appendix.\\

\begin{definition}
We denote by $\mathbb{I}:\mathfrak{G}\rightarrow{\mathfrak{G}^{*}}$ the mapping $\mathbb{I}(\mathbf{V})(\mathbf{W})=\ll \mathbf{V},\mathbf{W}\gg_{\mathbf{e}}$.\\
We define the Klein form on $\mathfrak{G}$ by integrating the klein formula given in item (12) in definition 12 
$$\mathfrak{K}(\mathbf{V},\mathbf{W})=\int_{\mathcal{L}}\mathfrak{k}(\mathbf{V}(z),\mathbf{W}(z))dz.$$
We have then $\mathfrak{K}(Ad_{\mathbf{g}}(\mathbf{V}),Ad_{\mathbf{g}}(\mathbf{W}))=\mathfrak{K}(\mathbf{V},\mathbf{W})$ for all $\mathbf{g}\in G$ and $\mathbf{V},\mathbf{W}\in\mathfrak{G}$, so we obtain that $\forall{\mathbf{U,V,W}\in\mathfrak{G}}$
\begin{equation}\label{eq:skewsym}
\mathfrak{K}(ad_{\mathbf{U}}\mathbf{V},\mathbf{W})=-\mathfrak{K}(ad_{\mathbf{U}}\mathbf{W},\mathbf{V}). \end{equation}
We define the inertia operator $A:\mathfrak{G}\rightarrow{\mathfrak{G}}$ such that $$\mathfrak{K}(A(\mathbf{V}),\mathbf{W})=\ll \mathbf{V},\mathbf{W}\gg_{e},$$ explicitly we have the formula $$A(\mathbf{V})(z)(p_{0}(z))=I_{z}(\omega_{\mathbf{V}(z)}),$$ and $$\omega_{A(\mathbf{V})(z)}=m_{z}\mathbf{V}(z)(p_{0}(z)).$$
\end{definition}

\section{Euler Poincare equations}
We defined the configuration space and identified him with a Lie group depending on the choice of a reference configuration, now we introduce a potential energy that quantify the obstruction of the motion to be a rigid one. \\ 
There is a natural vector field $X$ on $G$ defined by $$X(\mathbf{g})(z)=\partial_{z}\mathbf{g}(z)\in T_{\mathbf{g}(z)}\mathrm{SE}(3),$$
the zeros of $X$ are exactly rigid body motions of the snake robot, we denote by $$\xi:G\rightarrow{\mathfrak{G}},$$
the map $$\xi(\mathbf{g})=\theta_{L}(X(\mathbf{g})).$$
The following definition is natural.
\begin{definition}
Let $\mathcal{H}:\mathfrak{G}\rightarrow{\mathfrak{G}}$ be a symmetric positive definite (with respect to $\ll,\gg_{\mathbf{e}}$) endomorphism, we call elastic energy the function \begin{equation}\label{eq:pot}
U(\mathbf{g})=\frac{1}{2}\ll\mathcal{H}(\xi(\mathbf{g})),\xi(\mathbf{g})\gg_{\mathbf{e}}.
\end{equation}
This function satisfies $U(\mathbf{g})=0$ if and only if $\mathbf{g}$ is a rigid motion.
\end{definition}

A central concept in mechanics is the Lagrangian, that is defined by the difference between kinetic and potential energy :
\begin{definition}
The Lagrangian of the snake robot is given by 
\begin{equation}\label{eq:lag}
L(\mathbf{g},\mathbf{v})=\frac{1}{2}\ll \mathbf{v},\mathbf{v}\gg_{\mathbf{g}}-U(\mathbf{g}),
\end{equation}
for all $(\mathbf{g},\mathbf{v})\in TG$.
\end{definition}
Let $\mathbf{g}(t)\in G$ be the motion of a Snake robot, the principle of least action of physics tell that it minimizes the Lagrangian $L$ along all curves that fixes $\mathbf{g_{0}}$ and $\mathbf{g_{1}}$ which corresponds to initial and final positions.\\
\begin{definition}
We call a perturbation of a motion $\mathbf{g}(t)\in G$ a map $\mathbf{h}(s,t)\in G$ such that : 
\begin{itemize}
\item $\mathbf{h}(0,t)=\mathbf{g}(t)$ for all $t\in [0,1]$.
\item $\mathbf{h}(s,0)=\mathbf{g_{0}}$, $\mathbf{h}(s,1)=\mathbf{g_{1}}$, that is $\mathbf{h}(s,t)$ fixes the initial and final positions. 
\item We call the variation vector field $\mathbf{V}(t)=\partial_{s}\mathbf{h}(0,t)\in T_{\mathbf{g}(t)}G$ along $\mathbf{g}(t)$ and it satisfies $\mathbf{V}(0)=\mathbf{V}(1)=0$.
\end{itemize}
\end{definition}
The principle of least action tells us that 
\begin{equation}
    \int_{0}^{1}L(\mathbf{g}(t),\mathbf{g}'(t))dt=\min_{\gamma\in C^{\infty}([0,1],G)|\mathbf{\gamma}(0)=\mathbf{g}_{0}|\mathbf{\gamma}(1)=\mathbf{g}_{1}}\int_{0}^{1}L(\mathbf{\gamma}(t),\mathbf{\gamma}'(t))dt
\end{equation}
We compute the Lagrangian on these perturbations 
$$f(s):=\int_{0}^{1}L(\mathbf{h}(s,t),\partial_{t}\mathbf{h}(s,t))dt=\int_{0}^{1}\frac{1}{2}\ll\partial_{t}\mathbf{h}(s,t),\partial_{t}\mathbf{h}(s,t)\gg_{\mathbf{h}(s,t)}-U(\mathbf{h}(s,t))dt.$$
The principle of least action implies that $$f'(0)=0,$$ for all variations $\mathbf{V}\in\Gamma(\mathbf{g})$ such that $\mathbf{V}(0)=\mathbf{V}(1)=0$. \\
Using proprieties of covariant derivative formula (16) in proposition 4, we obtain  
$$\int_{0}^{1}\ll D_{s}\partial_{t}\mathbf{h}(0,t),\partial_{t}\mathbf{h}(0,t)\gg_{\mathbf{h}(0,t)}-dU_{\mathbf{h}(0,t)}(\partial_{s}\mathbf{h}(0,t))dt=0,$$
using (1) (3) in definition 11, we have
$$\int_{0}^{1}\ll D_{t}\mathbf{V}(t),\mathbf{g}'(t)\gg_{\mathbf{g}(t)}-dU_{\mathbf{g}(t)}(\mathbf{V}(t))dt=0,$$
by formula (16), we get $$\frac{d}{dt}\ll \mathbf{V}(t),\mathbf{g}'(t)\gg_{\mathbf{g}(t)}=\ll D_{t}\mathbf{V}(t),\mathbf{g}'(t)\gg_{\mathbf{g}(t)}+\ll \mathbf{V}(t),D_{t}\partial_{t}\mathbf{g}(t)\gg_{\mathbf{g}(t)}.$$
Using the fact that $\mathbf{V}(0)=\mathbf{V}(1)=0$ we get
\begin{equation}\label{eq:varform}
\int_{0}^{1}\ll D_{t}\partial_{t}\mathbf{g}(t),\mathbf{V}(t)\gg_{\mathbf{g}(t)}+dU_{\mathbf{g}(t)}(\mathbf{V}(t))dt=0,\end{equation}
for all $\mathbf{V}\in\Gamma(\mathbf{g})$ such that $\mathbf{V}(0)=\mathbf{V}(1)=0$.\\ \\
We now have to compute the differential of $U$.
\begin{lemma}
The differential of $U$ is given by 
\begin{equation}\label{eq:diffU}
dU_{\mathbf{g}}(\mathbf{V})=\ll\partial_{z}\mathbf{Z},\mathcal{H}(\xi(\mathbf{g}))\gg_{\mathbf{e}}+\ll[\xi(\mathbf{g}),\mathbf{Z}],\mathcal{H}(\xi(\mathbf{g})))\gg_{\mathbf{e}} ,\end{equation} where $\mathbf{Z}=\theta_{L}(\mathbf{V})$. 
\end{lemma}
\begin{proof}
Let $\mathbf{g}(t)\in G$ a curve such that $\mathbf{g}(0)=g=\begin{pmatrix}R_{\mathbf{g}}&u_{\mathbf{g}}\\ 0&1\end{pmatrix}$, $\mathbf{g}'(0)=\mathbf{V}=\begin{pmatrix}R_{\mathbf{g}}j(\omega_{\mathbf{Z}})& \mathbf{v}\\ 0 &0\end{pmatrix}$, $$\frac{d}{dt}U(\mathbf{g}(t))=<<\partial_{t}\xi(\mathbf{g}(t)),\mathcal{H}(\xi(\mathbf{g}(t)))>>.$$ 
We have $$\xi(\mathbf{g})=\begin{pmatrix}R_{\mathbf{g}}^{-1}\partial_{z}R_{\mathbf{g}}&R_{\mathbf{g}}^{-1}\partial_{z}u_{\mathbf{g}}\\ 0& 0\end{pmatrix}=\begin{pmatrix}j(\omega_{\xi(\mathbf{g})})&v_{\xi(\mathbf{g})}\\ 0&0\end{pmatrix}.$$
We compute $$\partial_{t}|_{t=0}\xi(\mathbf{g}(t))=\begin{pmatrix}-R_{\mathbf{g}}^{-1}\partial_{t}R_{\mathbf{g}}R_{\mathbf{g}}^{-1}\partial_{z}R_{\mathbf{g}}&-R_{\mathbf{g}}^{-1}\partial_{t}R_{\mathbf{g}}R_{\mathbf{g}}^{-1}\partial_{z}u_{\mathbf{g}}\\ 0&0\end{pmatrix}+\begin{pmatrix}R_{\mathbf{g}}^{-1}\partial_{z}\partial_{t}R_{\mathbf{g}}&R_{\mathbf{g}}^{-1}\partial_{z}\partial_{t}u_{\mathbf{g}}\\ 0&0\end{pmatrix}$$
$$=\begin{pmatrix}-j(\omega_{\mathbf{Z}})j(\omega_{\xi(\mathbf{g})})+R_{\mathbf{g}}^{-1}\partial_{z}(R_{\mathbf{g}}j(\omega_{\mathbf{Z}}))&-j(\omega_{\mathbf{Z}})v_{\xi(\mathbf{g})}+R_{\mathbf{g}}^{-1}\partial_{z}\mathbf{v} \\ 0 & 0\end{pmatrix}$$
$$=\begin{pmatrix}-j(\omega_{\mathbf{Z}})j(\omega_{\xi(\mathbf{g})})+j(\omega_{\xi(\mathbf{g})})j(\omega_{\mathbf{Z}})+\partial_{z}j(\omega_{\mathbf{Z}})& -j(\omega_{\mathbf{Z}})v_{\xi(\mathbf{g})}+R_{\mathbf{g}}^{-1}\partial_{z} \mathbf{v}\\ 0&0\end{pmatrix}.$$

Let $$\mathbf{Z}=\theta_{L}(\mathbf{V})=\begin{pmatrix}j(\omega_{\mathbf{Z}})&R_{\mathbf{g}}^{-1}\mathbf{v}\\ 0&0\end{pmatrix},$$ we obtain
$$\partial_{z}\mathbf{Z}=\begin{pmatrix}\partial_{z}j(\omega_{\mathbf{Z}})&-R_{\mathbf{g}}^{-1}\partial_{z}R_{\mathbf{g}}R_{\mathbf{g}}^{-1}\mathbf{v}+R_{\mathbf{g}}^{-1}\partial_{z}\mathbf{v}\\ 0& 0\end{pmatrix}.$$
Which is equal to $$\partial_{z}\mathbf{Z}=\begin{pmatrix}\partial_{z}j(\omega_{\mathbf{Z}})&-j(\omega_{\xi(\mathbf{g})})R_{\mathbf{g}}^{-1}\mathbf{v}+R_{\mathbf{g}}^{-1}\partial_{z}\mathbf{v}\\ 0 &0\end{pmatrix}.$$
This gives $$\partial_{t=0}\xi(\mathbf{g}(t))=\partial_{z}\mathbf{Z}+[\xi(\mathbf{g}),\mathbf{Z}],$$
and finishes the proof. 
\end{proof}
\begin{corollary}
The differential of $U$ is also given by $$dU_{\mathbf{g}}(\mathbf{V})=\ll\mathbb{I}^{-1}ad_{\xi(\mathbf{g})}^{*}\mathbb{I}\mathcal{H}(\xi(\mathbf{g})),\mathbf{Z}\gg_{\mathbf{e}}+<\mathcal{H}(\xi(\mathbf{g})),\mathbf{Z}>_{z}|^{L}_{0}-\ll\partial_{z}\mathcal{H}(\xi(\mathbf{g})),\mathbf{Z}\gg_{\mathbf{e}},$$
where $f|_{0}^{L}=f(L)-f(0)$. 
\end{corollary}

\begin{proof}
We observe that $$\ll[\xi(\mathbf{g}),\mathbf{Z}],\mathcal{H}(\xi(\mathbf{g}))\gg_{\mathbf{e}}=\mathbb{I}\mathcal{H}(\xi(\mathbf{g})))(ad_{\xi(\mathbf{g})}\mathbf{Z})$$ $$=ad_{\xi(\mathbf{g})}^{*}(\mathbb{I}\mathcal{H}(\xi(\mathbf{g})))(\mathbf{Z})=\ll\mathbb{I}^{-1}ad_{\xi(\mathbf{g})}^{*}\mathbb{I}\mathcal{H}(\xi(\mathbf{g})),\mathbf{Z}\gg_{\mathbf{e}}.$$
On the other hand, an integration by parts allow us to get 
$$\ll\partial_{z}\mathbf{Z},\mathcal{H}(\xi(\mathbf{g}))\gg_{\mathbf{e}} =<\mathcal{H}(\xi(\mathbf{g})),\mathbf{Z}>_{z}|^{L}_{0}-\ll\partial_{z}\mathcal{H}(\xi(\mathbf{g})),\mathbf{Z}\gg_{\mathbf{e}},$$
which finishes the proof.
\end{proof}\\\\
Return now to the expression \eqref{eq:varform}
$$\int_{0}^{1}\ll D_{t}\partial_{t}\mathbf{g}(t),\mathbf{V}(t)\gg_{\mathbf{g}(t)}+dU_{\mathbf{g}(t)}(\mathbf{V}(t))dt=0,$$ for all $\mathbf{V}\in\Gamma(\mathbf{g})$ such that $\mathbf{V}(0)=\mathbf{V}(1)=0$. \\ \\
If we denote by $\mathbf{W}(t)=\theta_{L}(\mathbf{g}'(t))$ and $\mathbf{Z}(t)=\theta_{L}(\mathbf{V}(t))$. \\\\
We obtain by using the left invariance of $\ll,\gg$ and formula (10)
$$\int_{0}^{1}\ll\partial_{t}\mathbf{W}-\mathbb{I}^{-1}ad_{\mathbf{W}(t)}^{*}\mathbb{I}\mathbf{W}(t),\mathbf{Z}(t)\gg_{\mathbf{e}}$$ $$+\ll\mathbb{I}^{-1}ad_{\xi(\mathbf{g}(t))}^{*}\mathbb{I}\mathcal{H}(\xi(\mathbf{g}(t))),\mathbf{Z}(t)\gg_{\mathbf{e}}-\ll\partial_{z}\mathcal{H}(\xi(\mathbf{g}(t))),\mathbf{Z}(t)\gg_{\mathbf{e}}$$ $$+<\mathcal{H}(\xi(\mathbf{g}(t))),\mathbf{Z}(t)>_{z}|^{L}_{0}dt=0,$$
for all $\mathbf{Z}(t)\in\mathfrak{G}$ such that $\mathbf{Z}(0)=\mathbf{Z}(1)=0$.\\ \\
We choose variations such that $\mathbf{Z}(t)(0)=\mathbf{Z}(t)(L)=0$ and apply Dubois-Raymond Lemma to obtain
\begin{equation}\label{eq:motion}
\partial_{t}\mathbf{W}-\mathbb{I}^{-1}ad_{\mathbf{W}}^{*}\mathbb{I}\mathbf{W}+\mathbb{I}^{-1}ad_{\xi(\mathbf{g})}^{*}\mathbb{I}\mathcal{H}(\xi(\mathbf{g}))-\partial_{z}\mathcal{H}(\xi(\mathbf{g}))=0,
\end{equation} and the boundary conditions 
\begin{equation}\label{eq:bd}
\mathcal{H}(\xi(\mathbf{g}))(0)=\mathcal{H}(\xi(\mathbf{g}))(L)=0.\end{equation}
Using the fact that $ad\in L(\mathfrak{G},L_{skew}(\mathfrak{G},\mathfrak{K}))$ (equation \eqref{eq:skewsym}) we obtain $\forall{\mathbf{u,v}\in\mathfrak{G}}$
$$\mathbb{I}^{-1}ad_{\mathbf{u}}^{*}\mathbb{I}\mathbf{v}=A^{-1}[A(\mathbf{v}),\mathbf{u}],$$ 
this reduction allow us to have the more simpler equations
\begin{equation}
\partial_{t}A(\mathbf{W})+[\mathbf{W},A(\mathbf{W})]=A\partial_{z}\mathcal{H}(\xi(\mathbf{g}))+[\xi(\mathbf{g}),A\mathcal{H}(\xi(\mathbf{g}))],\end{equation}
with boundary conditions $$\mathcal{H}(\xi(\mathbf{g}))(0)=\mathcal{H}(\xi(\mathbf{g}))(L)=0,$$
We have proved the following theorem.
\begin{theorem}
The dynamic of a snake robot is described by the following quasi-linear hyperbolic PDEs.
\begin{equation}\label{eq:dyn}
\partial_{t}\theta_{L}(\partial_{t}\mathbf{g})+A^{-1}[(\theta_{L}(\partial_{t}\mathbf{g}),A((\theta_{L}(\partial_{t}\mathbf{g}))]=\partial_{z}\mathcal{H}(\theta_{L}(\partial_{z}\mathbf{g}))+A^{-1}[\theta_{L}(\partial_{z}\mathbf{g}),A\mathcal{H}(\theta_{L}(\partial_{z}\mathbf{g}))],
\end{equation}
with initial conditions $$\mathbf{g}(0)=\mathbf{g_{0}}\in C^{\infty}(\mathcal{L},\mathrm{SE}(3)),$$ $$\theta_{L}(\partial_{t}\mathbf{g})(0)=\mathbf{W_{0}}\in C^{\infty}(\mathcal{L},\mathfrak{se}(3)),$$ and the boundary value conditions $$\partial_{z}\mathbf{g}(t,L)=\partial_{z}\mathbf{g}(t,0)=0.$$
\end{theorem}
\begin{remark} The above theorem describes the uncontrolled motion of a snake robot using a Cosserat beam model. It has been stated in coordinates in \cite{BLCR-2022}. Note that the Lie Brackets there should be replaced by $ad$ and $ad^T$ symbols to be completely intrinsic. 
\end{remark}
The following result is fundamental in order to later prove existence (and possibly uniqueness) of solutions in appropriate functional spaces. 
\begin{theorem}{Conservation of energy}\\
The dynamics of theorem 4.2 conserves the total energy 
\begin{equation}
    E_{T}(t)=\frac{1}{2}\ll\mathbf{g}'(t),\mathbf{g}'(t)\gg_{\mathbf{g}(t)}+U(\mathbf{g}(t)).
\end{equation}

\end{theorem}
\begin{proof}
We have $$\frac{d}{dt}E_{T}(t)=\ll D_{t}\mathbf{g}'(t),\mathbf{g}'(t)\gg_{\mathbf{g}(t)}+dU_{\mathbf{g}(t)}(\mathbf{g}'(t))$$ which is 0 from (10) and Dubois-Raymond Lemma.
\end{proof}
\section{Adding a control law}
As we can see in the two terms $\partial_{t}\theta_{L}(\partial_{t}\mathbf{g})$ and $\partial_{z}\mathcal{H}(\theta_{L}(\partial_{z}\mathbf{g}))$ in \eqref{eq:motion}, there is a hyperbolic evolution which agrees with the observation of the evolution of the snake.
However, \eqref{eq:motion} is still related to a passive rod. In order to relate it to a snake, one need to change its passive constitutive law into an active one by adding a model of the control.

At its simplest, such a control can be introduced into the model by replacing, in the right-hand side \eqref{eq:motion}, the passive elastic contribution of $\mathcal{H}(\theta_{L}(\partial_{z}\mathbf{g}))$, by:

\begin{equation}\label{control_law}
\mathcal{H}(\theta_{L}(\partial_{z}\mathbf{g}))+u(\theta_{L}(\partial_{z}\mathbf{g}),\partial_{t}(\theta_{L}(\partial_{z}\mathbf{g})))
\end{equation}
where $u$ represents a local feedback control term. Such a term models the effects of the snake's muscles or the internal actuation of a snake-like robot. Note that in general, this term can include an autonomic feed-forward component modelling the nervous activity of a central pattern generator (CPG), or a feedback on $g$ via the snake's vestibular organs (which inform it of its orientation relative to gravity). More ambitiously, it can define a non-local feedback where $z$ in $u$ is replaced by any other value along the snake's spine. These non-local terms model the effects of a distributed action (muscles connecting distant vertebrae), or tactile sensors covering its skin whose information is processed in a distributed way. At last, the passive contribution $\mathcal{H}(\theta_{L}(\partial_{z}\mathbf{g})$ can be itself replaced by any left-invariant function of the configuration. In this case, the linear elastic constitutive law assumed throughout the previous developments can be changed in any 
nonlinear constitutive law modelling the complex rheology of muscles and tendons.
\begin{remark}
Note that modifying \eqref{eq:motion} by using \eqref{control_law} requires to take the partial derivative of the control with respect to the space variable $z$. Hence from a control theoretical perspective, the latter should be the control function.  
\end{remark}

 \section{Appendix} 
 Here we give an intuitive explanation of the rigorous definitions that we developed on the presentation.\\
 We distinguish between three parts, the first one deals with the insufficiency of classical medium to describe snake robots with rigid motion of the sections. The second one deals with the Riemannian structure of $G$, and the third one recalls basic results about Riemannian geometry on Lie groups.

 \subsection{Insufficiency of classical medium}
A preliminary approach could be the classical one, which models the snake robot as a cylinder $\mathcal{B}=\left\{(x,y,z)\in\mathbb{R}^{3}, x^{2}+y^{2}\leq R^{2}, 0\leq z\leq L\right\}$, a configuration of  the snake robot is an embedding $p:B\rightarrow{\mathbb{R}^{3}}$. \\ \\
The configuration space is then the set of all embedding [11] of $\mathcal{B}$ denoted $Emb(\mathcal{B})$, in order to simplify the study, we impose that the circular sections of the cylinder have a rigid motion, in this case we see that knowing motion of $p(0,0,z)$ for $0\leq z\leq L$, we know motion of $p(x,y,z)$ for $(x,y,z)\in \mathcal{B}$.\\ \\
We call a section the set $S_{z}=\left\{(x,y,w)\in\mathcal{B}, w=z\right\}$, and we see that for a classical medium we have $p(S_{z})\subset (p(0,0,z)')^{\perp}$, and this imply clearly that if $p(0,0,z)$ is curved, the sections that are close have to cut each other, and this would not give an embedding. \\ \\
In order to allow a rigid motion to the "sections", we have to introduced the notion of principal fiber bundle in definition 1.\\ \\
Now we give an intuitive explanation of why we let the "sections" have all possible rigid motion, if we take a rule and fix it on one of the extremities, we see that the degree of freedom of the other extremity is 1. For a cylinder, there is 2 degrees of freedom because of the torsion. A Cosserat medium is a medium where we have the 3 degrees of freedom which are Euler angles and which parameterize the rotation of sections in $SO(3)$.  
\subsection{The kinetic energy of a Cosserat medium}
We have considered a Riemannian structure in the Lie group $G$ that is induced by the kinetic energy of the snake robot, but what is the kinetic energy for a Cosserat medium ? In order to answer this question, it is instructive to recall the definition of kinetic energy for classical medium. \\\\
A classical medium is 3 dimensional compact manifold [11] $\mathcal{B}$ called body, furnished with a measure mass $\mu$, and a configuration is an embedding of $\mathcal{B}$ in $\mathbb{R}^{3}$. The configuration space is then $Emb(\mathcal{B})$ the space of all embeddings of $\mathcal{B}$. \\
We call $\Omega=p(\mathcal{B})$ the actual configuration, and we consider in this manifold the measure $p_{*}\mu$ (it is the unique measure mass on $\Omega$ such that $p$ preserves mass). 
Consider now a motion $p(t)\in Emb(\mathcal{B})$ of the body, We define the Lagrangian velocity by $V(t,X)=\partial_{t}p(t)(X)$ for $X\in\mathcal{B}$, and the Eulerian velocity by $u(t,x)=V(t,p(t)^{-1}(x))$ where $x=p(t)(X)$
$$E_{k}(t)=\frac{1}{2}\int_{\Omega} u^{2}p_{*}\mu,$$
by the changing variable formula we get $$E_{k}(t)=\frac{1}{2}\int_{B}\partial_{t}p^{2}\mu,$$ which have sens even if $p$ is not an embedding. This suggests to use Lagrangian velocity $V$ to compute kinetic energy of our snake robot. \\
We return to the cylinder $\mathcal{B}=\left\{(x,y,z)\in\mathbb{R}^{3},x^{2}+y^{2}\leq R^{2},0\leq z\leq L\right\}$, and for $\mathbf{g}\in G$ consider the smooth mapping $p:\mathcal{B}\rightarrow{\mathbb{R}^{3}}$ defined by 
$$p(x,y,z)=\mathbf{g}(z)(x,y,z).$$
The kinetic energy is then $$E_{k}(t)=\frac{1}{2}\int_{\mathcal{B}}|\partial_{t}\mathbf{g}(t)(z)(x,y,z)|^{2}\rho_{0}(z)dxdydz,$$
and using the sections $S_{z}$ we obtain 
$$E_{k}(t)=\frac{1}{2}\int_{\mathcal{L}}(\int_{S_{z}}|\partial_{t}\mathbf{g}(t)(z)(x,y,z)|^{2}\rho_{0}(z)dxdy)dz.$$
Now we use the fact that $\mathbf{g}(t)\in G$ to show that $$\int_{S_{z}}|\partial_{t}\mathbf{g}(t)(z)(x,y,z)|^{2}\rho_{0}dxdy=2\pi R^{2}\rho_{0}(z)|\mathbf{W}(t)(z)(0,0,z)|^{2}$$ $$+\omega_{\mathbf{W}(t)(z)}.I_{z}(\omega_{\mathbf{W}(t)(z)})dz ,$$
where $\mathbf{W}(t)=\theta_{L}(\partial_{t}\mathbf{g}(t))$. Putting $m_{z}=2\pi R^{2}\rho_{0}(z)$, we obtain $$E_{k}(t)=\frac{1}{2}\int_{\mathcal{L}}m_{z}|\mathbf{W}(t)(z)(0,0,z)|^{2}+\omega_{\mathbf{W}(t)(z)}.I_{z}(\omega_{\mathbf{W}(t)(z)}).$$
Which is exactly the formula \eqref{eq:kinetic}
\subsection{Basics on Riemannian geometry}
A central concept in Riemannian geometry is parallel transportation.\\
Let $(M,g)$ be a Riemannian manifold, for a curve $\gamma:[0,1]\rightarrow{M}$, we take a tangent vector at $\gamma(0)$ and we let him evolve continuously such that it preserves his length and his angle with $\gamma'$, we denote $P(\gamma(1),\gamma(0)):T_{\gamma(0)}M\rightarrow{T_{\gamma(1)}M}$ by the parallel transport. \\
The Riemannian connection is the corresponding structure for whom the flow is the parallel transportation i.e we have $$\nabla_{X}Y(p)=\frac{d}{dt}|_{t=0}P(\gamma(0),\gamma(t))(Y(\gamma(t))),$$
where $\gamma(0)=p$ and $\gamma'(0)=X(p)$.\\ 
It can be shown [1] that this connection is characterised by Koszul formula 
\begin{theorem}
There exists a unique linear connection $\nabla :\Gamma(TM)\times\Gamma(TM)\rightarrow{\Gamma(TM)}$ such that for each $f\in C^{\infty}(M)$
\begin{itemize}
    \item $\nabla_{fX}Y=f\nabla_{X}Y$.
    \item $\nabla_{X}fY=df(Y)+f\nabla_{X}Y$. 
    \item $\nabla_{X}Y-\nabla_{Y}X=[X,Y]$.
     \item $L_{X}g(Y,Z)=g(\nabla_{X}Y,Z)+g(Y,\nabla_{X}Z)$.
\end{itemize}
which is characterized by the Koszul formula 
\begin{equation}\label{eq:koszul}
g(\nabla_{X}Y,Z)=\frac{1}{2}(L_{X}g(Y,Z)+L_{Y}g(X,Z)-L_{Z}g(X,Y)+g([X,Y],Z)+g([Z,X],Y)-g([Y,Z],X)).\end{equation}
\end{theorem}
The connection $\nabla$ induces a linear connection on vector fields that are tangent to $M$ along a curve $\gamma:[0,1]\rightarrow{M}$, the set of all these vector fields is denoted by $\Gamma(\gamma)$. \\We have the following result.
\begin{proposition}
There is a unique operator $\frac{D}{Dt}:\Gamma(\gamma)\rightarrow{\Gamma(\gamma)}$ such that : 
\begin{itemize}
    \item $\frac{DfV}{Dt}(t)=f'(t)V(t)+f(t)\frac{DV}{Dt}(t)$ for all $f\in C^{\infty}([0,1])$ and $V\in \Gamma(\gamma)$.
    \item If $X\in\Gamma(TM)$ we have $\frac{DX\circ \gamma}{Dt}(t)=\nabla_{\gamma'(t)}X(\gamma(t))$.
    \item We have the following formula \begin{equation}\label{eq:killmet}
    \frac{d}{dt}g_{\gamma(t)}(V(t),W(t))=g_{\gamma(t)}(\frac{DV}{Dt}(t),W(t))+g_{\gamma(t)}(V(t),\frac{DW}{Dt}(t)).\end{equation}
\end{itemize}
\end{proposition}
Now consider a perturbation $h(s,t)\in M$ of a curve $\gamma(t)\in M$, we have the following formula that is proved by computations in local coordinates.
\begin{lemma}
Let $\gamma(t)\in M$ be a curve, and $h(s,t)\in M$ be a perturbation, so we have 
$$\frac{D}{Dt}\frac{\partial}{\partial s}h(s,t)=\frac{D}{Ds}\frac{\partial}{\partial t}h(s,t).$$
\end{lemma}
\subsection{Lie Groups and Riemannian geometry}
Now let suppose that a Lie group $G$ furnished with a Riemannian metric that is left invariant, we denote by $\mathbb{I}=g_{e}$ the value of the Riemannian metric at the neutral element.\\ \\
When the metric is left invariant [2] [8], the Levi Civita connection is also left invariant.
\begin{lemma}
let $G$ be a Lie group furnished with a left invariant metric that is $\mathbb{I}$ at the neutral element, then $\nabla$ is left invariant and we have explicitly 
\begin{equation}\label{eq:lefti}
    \nabla_{V_{L}}W_{L}=(\frac{1}{2}[V,W]-\mathbb{I}^{-1}ad_{V}^{*}\mathbb{I}W-\mathbb{I}^{-1}ad_{W}^{*}\mathbb{I}V)_{L},\end{equation}
where $V,W\in\mathfrak{g}$ and $V_{L}(g)=T_{e}L_{g}(V)$.
\end{lemma}
\begin{proof}
We use Koszul formula \eqref{eq:koszul} for left invariant vector fields $V_{L}, W_{L}, U_{L}$ to show that $\nabla_{V_{L}}W_{L}$ is left invariant, using the fact that the Lie bracket commutes with the pullback and that the metric is left invariant. \end{proof}

This results allow us to simplify equations of motion of the snake robot

\begin{proposition}
Let $g:[0,1]\rightarrow{G}$ be a curve, 
\begin{equation}\label{eq:eulerarnold}
    \theta_{L}(\frac{Dg'}{Dt}(t))=\partial_{t}\theta_{L}(g'(t))-\mathbb{I}^{-1}ad_{\theta_{L}(g'(t))}^{*}\mathbb{I}\theta_{L}(g'(t)).\end{equation}
\end{proposition}
\begin{proof}
Let $e_{i}$ be a basis of $\mathfrak{g}$, and let $v^{i}$ be such that $$\theta_{L}(g'(t))=v^{i}(t)e_{i},$$
so we get $$g'(t)=v^{i}(t)(e_{i})_{L},$$
this gives $$\frac{Dg'}{Dt}(t)=\frac{dv^{i}}{dt}(e_{i})_{L}+v^{i}(t)\nabla_{g'(t)} (e_{i})_{L}(g(t)),$$
by properties of Levi Civita connection we obtain $$\theta_{L}(\frac{Dg'}{Dt}(t))=\frac{dv^{i}}{dt}(t)e_{i}+\frac{1}{2}v^{i}(t)v^{j}(t)([e_{i},e_{j}]-\mathbb{I}^{-1}ad_{e_{i}}^{*}\mathbb{I}e_{j}-\mathbb{I}^{-1}ad_{e_{j}}^{*}\mathbb{I}e_{i}),$$ 
and the conclusion follows.
\end{proof}
\section{Notations and definitions}
We give here tools that is needed to understand the presentation that follows, we begin by basic definitions in abstract algebra and linear algebra.
\begin{definition}{Algebraic tools}
\begin{enumerate}[label=\arabic*)]
    \item Let $G,H$ be two groups, we call a morphism of group from $G$ to $H$ a map that preserves the group structure i.e $f:G\rightarrow{H}$ such that $f(x.y)=f(x).f(y)$ $\forall{x,y\in G}$.
    \item Let $X$ be a set, we call $S(X)$ the group of permutations of $X$, it consists on the maps $f:X\rightarrow{X}$ that are bijective, the composition law is the usual composition of functions, the neutral element being the identity mapping of $X$.
    \item A group action of a group $G$ on a set $X$ is a map $*:G\times X\rightarrow{X}$ that satisfies :\\
    1. $\forall{g,h\in G}$ $g*(h*x)=(g.h)*x$.\\
    2. $\forall{x\in X}$ $e*x=x$.\\
    We associate to each group action a morphism of group $$\phi:G\rightarrow{S(X)}$$ defined by $\phi(g)(x)=g*x$.\\ 
    The action is free if this morphism is injective.\\
    The action is transitive if $\forall{x,y\in X}$ there exists $g\in G$ such that $g*x=y$.\\
    When a group action is free and transitive, we can identify $X$ to $G$ by the following way. We fix $x\in X$, and we consider the bijection $\phi_{x}:g\in G\rightarrow{g*x\in X}$. The element $x\in X$ plays the role of the reference configuration in mechanics. 
    \item A linear map from a vector space $V$ to $V$ is a map $f:V\rightarrow{V}$ such that $$f(\alpha.x+\beta.y)=\alpha.f(x)+\beta.f(y).$$
    \item A scalar product in vector space $V$ is a function $(,):V\times V\rightarrow{\mathbb{R}}$ such that :\\
    1. $\forall{x,y,z\in V}$ $\forall{\alpha,\beta\in\mathbb{R}}$ we have $(\alpha x+\beta y,z)=\alpha(x,z)+\beta(y,z)$.\\ 
    2. $\forall{x\in V-{0}}$, $(x,x)>0$.\\
    3. $\forall{x,y\in V}$, $(x,y)=(y,x)$.
\item A Lie algebra is a vector space $\mathfrak{g}$ furnished with a Lie bracket $[,]:\mathfrak{g}\times\mathfrak{g}\rightarrow{\mathfrak{g}}$ that satisfies :\\
1. $\forall{x,y,z\in\mathfrak{g}}$, $\forall{\alpha,\beta\in\mathbb{R}}$, $[\alpha.x+\beta.y,z]=\alpha.[x,z]+\beta.[y,z]$.\\
2. $\forall{x,y\in\mathfrak{g}}$, $[x,y]=-[y,x]$.\\
3. $\forall{x,y,z\in\mathfrak{g}}$, $[x,[y,z]]+[y,[z,x]]+[z,[x,y]]=0$.
\item In a vector space $V$ furnished with a bilinear symmetric form $\mathfrak{k}$, we denote by $L_{skew}(V,\mathfrak{k})$ the space of skew-symmetric endomorphism with respect to $\mathfrak{k}$, that is $f:V\rightarrow{V}$ such that \begin{equation}
    \label{eq:La}
    \mathfrak{k}(f(u),v)=-\mathfrak{k}(u,f(v)).
\end{equation}
\end{enumerate}
\end{definition}
Now we introduce some elementary differential geometric tools and combine them with the algebraic tools, a good reference is \cite{Lee}.
\begin{definition}{Differential geometric tools}
\begin{enumerate}[label=\arabic*)]
   
\item An embedding of a manifold $B$ is a map $p:B\rightarrow{\mathbb{R}^{3}}$ that is smooth and is a diffeomorphism on its image $p(B)$.
   
\item A submersion from $M$ to $N$ is a smooth map $\phi:M\rightarrow{N}$ such that the tangent linear map $d\phi_{x}:T_{x}M\rightarrow{T_{\phi(x)}N}$ is surjective for each $x\in M$. \\A fundamental result of submersions is that $f^{-1}(\left\{a\right\})$ is a submanifold of $M$ of dimension $dim(M)-dim(N)$ for each $a\in f(M)$, and we have $T_{x}f^{-1}(\left\{a\right\})=Ker(df_{x})$ $\forall{x\in f^{-1}(\left\{a\right\})}$. 
\item For two manifolds $M,N$ and a diffeomorphism $\phi:M\rightarrow{N}$, for $Y\in \Gamma(TN)$, we define the pullback by $\phi^{*}Y\in\Gamma(TM)$ $$\phi^{*}Y(x)=d\phi^{-1}_{\phi(x)}(Y\circ\phi(x)).$$
\item Let $M$ be a manifold and $f\in C^{\infty}(M)$ and $X\in\Gamma(TM)$, the lie derivative of $f$ on the direction of $X$ is $$L_{X}f(x)=d_{x}f(X(x))$$ and for two elements $X,Y\in\Gamma(TM)$, we define the Lie bracket of vector fields $[X,Y]$ as the unique vector field such that $$L_{[X,Y]}f=L_{X}L_{Y}f-L_{Y}L_{X}f,$$ for each $f\in C^{\infty}(M)$.  

\item Let $\phi:M\rightarrow{N}$ be a diffeomorphism, we have that for all vector fields $X,Y\in\Gamma(TN)$ $$\phi^{*}[X,Y]=[\phi^{*}X,\phi^{*}Y].$$
\item In a manifold $M$ $\delta x$ is an element of $T_{x}M$ for $x\in M$.

\item The adjoint representation of $G$ is defined by 
$$Ad_{g}:\mathfrak{g}\rightarrow{\mathfrak{g}},$$ $$v\rightarrow{d(i_{g})_{e}(v)} ,$$ where $i_{g}$ is the interior automorphism $i_{g}(h)=ghg^{-1}$.\\
In fact $Ad:G\rightarrow{GL(\mathfrak{g})}$ is a morphism of group. 
\item The adjoint representation of $\mathfrak{g}$ is the map $$ad_{u}:\mathfrak{g}\rightarrow{\mathfrak{g}},$$ defined by  $$ad_{u}v=dAd_{e}(u)(v)$$ for each $u,v\in\mathfrak{g}$, in fact $ad:u\in\mathfrak{g}\rightarrow{ad_{u}\in L(\mathfrak{g})}$.\\
The Lie bracket that endows $\mathfrak{g}$ with a structure of Lie algebra is $[u,v]:=ad_{u}v$. 
\item For a compact manifold, the Lie algebra of the group $G=\mathrm{Diff}(M)$ is $\Gamma(TM)$, the Lie bracket introduced in the precedent item is the Lie bracket of vector fields. 
\item We can define the Lie bracket of two elements in $\mathfrak{g}$ by the following process :\\
1. We extend $u,v\in\mathfrak{g}$ on $TG$ such that it remains left invariant i.e we consider $u_{L},v_{L}$ two vector fields defined by $u_{L}(g)=d(L_{g})_{e}(u)$ and same for $v$. \\
2. We compute the Lie bracket of $u_{L},v_{L}$ and evaluate at $e$, $[u,v]=[u_{L},v_{L}](e)$.\\
By the fact that pullback commute with Lie bracket, the element $[u,v]$ is well defined, and it can be easily proved that $[u,v]=ad_{u}v$.
\item In a Lie group $G$ we call the left Maurer-Cartan form the map $$\theta_{L}(v)=(d(L_{g})_{e})^{-1}(v),$$ where $L_{g}$ is the left multiplication by $g$. The Maurer-Cartan form is very useful to avoid computations on a nonlinear space that is the group $G$ and instead do them in $\mathfrak{g}$ that is a linear space. 
\item An action of a Lie group $G$ on a manifold $M$ is a morphisme of groups $\phi:G\rightarrow{\mathrm{Diff}(M)}$ from $G$ to the group of diffeomorphisms of $M$, we say that the action of free if this morphisme is injective, and is transitive if for all $x,y\in M$ there is $g\in G$ such that $\phi(g)(x)=y$.\\
If the action is free and transitive, we can identify $M$ to $G$. This identification depends on a fixed $x\in M$ and is done by the diffeomorphism $$\phi_{x}:g\in G\rightarrow{\phi(g)(x)\in M}.$$ The element $x\in M$ will play the role of reference configuration in mechanics.
\end{enumerate}
\end{definition}
\begin{definition}{Examples and applications}

\end{definition}
\begin{enumerate}[label=\arabic*)]
    \item We denote by $\mathbb{R}^{3}$ the euclidean space of dimension 3 furnished with the canonical scalar product $(,)$ and the vector product $\times$.
\item We denote by $M_n(\mathbb{R})$ the set of square matrices of order $n$ and by $M^\top$ the transpose of $M$ defined as: $M^\top_{i,j}=M_{j,i}$ for all $i,j\in \lbrace 1,\ldots,n\rbrace$. The matrix $I_n$ denotes the identity matrix.
\item We denote by $S_{n}(\mathbb{R})$ the set of symmetric matrices, i.e matrices such that $M^\top=M$.
\item We denote by $A_{n}(\mathbb{R})$ the set of skew symmetric matrices, i.e matrices such that $M^\top=-M$.
\item We denote by $I_{3}$ denote the neutral element of multiplication, i.e the matrix  $$I_{3}=\begin{pmatrix}1&0&0\\0&1&0\\0&0&1
\end{pmatrix}.$$
\item We denote by $\mathrm{SO}(3)=\left\{O\in M_{3}(\mathbb{R}) \mid O^{\top}.O=I_{3},\mathrm{det}(O)=1\right\}$, the elements of this set are rotations, it is a group for multiplication of matrices, and the multiplication of two consecutive rotations is the composition of the rotations.\\ 
The structure of manifold can be proved by observing that the map $$f:A\in GL_{3}(\mathbb{R})\rightarrow{A.A^\top\in S_{3}(\mathbb{R})},$$ 
is a submersion, we have $df_{A}(H)=A.H^\top+H.A^\top$ and then $T_{I_{3}}SO(3)=A_{3}(\mathbb{R})$. 
\item We denote by $\mathrm{SE}(3):=\left\{\begin{pmatrix}
R & u\\ 0 & 1
\end{pmatrix}\in M_4(\mathbb{R})\left|R \in \mathrm{SO}(3),\: u\in \mathbb{R}^3\right.\right\}$. It is a group for the composition law defined as follows:
$$\begin{pmatrix}R_{1}&u_{1}\\0&1\end{pmatrix}\begin{pmatrix}R_{2}&u_{2}\\0&1\end{pmatrix}=\begin{pmatrix}R_{1}R_{2}&u_{1}+R_{1}u_{2}\\0&1\end{pmatrix}.$$
The elements of this set are displacements. The composition law is the combination of two consecutive displacements, it's neutral element is $I_{4}$. \\
Any $g\in \mathrm{SE}(3)$ is denoted by $g:=\begin{pmatrix} R_g & u_g\\ 0 & 1 \end{pmatrix}$. A displacement can be seen as an application $g:\mathbb{R}^3\rightarrow \mathbb{R}^3$, which is an affine map defined as: $g(x):=R_g x + u_g$.
\item We denote by $\mathfrak{so}(3)$ the set of all skew-symmetric matrices furnished with the commutator $[A,B]=AB-BA$ of matrices, it is the Lie algebra of $SO(3)$, it is isomorphic to $\mathbb{R}^{3}$ by the canonical isomorphism $$c: (\mathfrak{so}(3),[,])\rightarrow{(\mathbb{R}^{3},\times)},$$ defined by $A.x=c(A)\times x$ for all $x\in\mathbb{R}^{3}$, for which inverse is $j:\mathbb{\mathbb{R}}^{3}\rightarrow{\mathfrak{so}(3)}$ is explicitly $$j\begin{pmatrix}\omega_{1}\\ \omega_{2}\\ \omega_{3}\end{pmatrix}=\begin{pmatrix}0& -\omega_{3}& \omega_{1}&\\ \omega_{3} &0 &-\omega_{2}  \\ -\omega_{1}& \omega_{2}& 0 \end{pmatrix}.$$ 
\item We denote by $\mathfrak{se}(3)=\left\{\begin{pmatrix}j(\omega)&v\\0&0\end{pmatrix}, \omega,v\in\mathbb{R}^{3}\right\}$, it is the Lie algebra of $SE(3)$, for $V\in\mathfrak{se}(3)$ we denote by $\omega_{V}$ the unique vector such that $V=\begin{pmatrix}j(\omega_{V})&v\\0&0\end{pmatrix}$ for some $v\in\mathbb{R}^{3}$.
\item The adjoint representation of $SE(3)$ is given by 
$$Ad_{g}\begin{pmatrix}j(\omega_{V})&v\\0&0\end{pmatrix}=\begin{pmatrix}Rj(\omega_{V})R^{-1}&-Rj(\omega_{V})R^{-1}u+Rv\\0&0\end{pmatrix},$$
where $g=\begin{pmatrix}R&u\\0&1\end{pmatrix}$.
\item The Lie bracket on $\mathfrak{se}(3)$ is given by 
$$[V,W]=ad_{V}W:=\begin{pmatrix}j(\omega_{V}\times\omega_{W})&j(\omega_{V})w-j(\omega_{W})v\\0&0\end{pmatrix}.$$
\item The Klein form $\mathfrak{k}:\mathfrak{se}(3)\times\mathfrak{se}(3)\rightarrow{\mathbb{R}}$ is defined by $$\mathfrak{k}(V,W)=v.\omega_{W}+w.\omega_{V},$$
by basic computations we have $$\mathfrak{k}(Ad_{g}(V),Ad_{g}(W))=\mathfrak{k}(V,W),$$ for $g\in \mathrm{SE}(3)$ and $V,W\in\mathfrak{se}(3)$.
The Klein form will be useful to simplify the equations of motion introducing inertia operator by the mean of kinetic energy. 
\item The tangent space of $\mathrm{SE}(3)$ at $g\in \mathrm{SE}(3)$ is given by $$T_{g}\mathrm{SE}(3)=\left\{\begin{pmatrix}R.j(\omega)&v\\0&0\end{pmatrix},\omega,v\in\mathbb{R}^{3}\right\},$$
$$d(L_{g})_{e}\begin{pmatrix}j(\omega)&v\\0&0\end{pmatrix}=\begin{pmatrix}R.j(\omega)& Rv\\0&0\end{pmatrix},$$

\item In the particular case of $\mathrm{SE}(3)$, the Maurer-Cartan form is 
$$\theta_{L}\begin{pmatrix}R.j(\omega)&v\\0&0\end{pmatrix}=\begin{pmatrix}j(\omega)&R^{-1}v\\0&0\end{pmatrix}.$$

\end{enumerate}

\end{document}